\documentclass[10pt,reqno]{amsart}
\usepackage{graphicx, amsmath, amssymb}

\newtheorem{Thm}{Theorem}[section]

\newtheorem{Prop}[Thm]{Proposition}

\newtheorem{Lem}[Thm]{Lemma}

\def\bB {\mathbf{B}}

\def\bC {\mathbf{C}}

\def\bL {\mathbf{L}}

\def\bR {\mathbf{R}}
\def\bS {\mathbf{S}}
\def\bT {\mathbf{T}}

\def\cH {\mathcal{H}}

\def\cK {\mathcal{K}}
\def\cL {\mathcal{L}}
\def\cM {\mathcal{M}}

\def\a {{\alpha}}

\def\g {{\gamma}}
\def\Ga {{\Gamma}}
\def\de {{\delta}}
\def\eps {{\varepsilon}}
\def\th {{\theta}}

\def\si {{\sigma}}

\def\om {{\omega}}
\def\Om {{\Omega}}

\def\rstr {{\big |}}
\def\indc {{\bf 1}}

\def\la {\langle}
\def\ra {\rangle}

\def\d {{\partial}}
\def\grad {{\nabla}}

\newcommand{\Supp}{\operatorname{supp}}
\newcommand{\Supess}{\operatorname{supess}}

\newcommand{\Dist}{\operatorname{dist}}

\newcommand{\Ker}{\operatorname{ker}}

\newcommand{\be}{\begin{equation}}
\newcommand{\ee}{\end{equation}}
\newcommand{\bmat}{\begin{matrix}}
\newcommand{\emat}{\end{matrix}}
\newcommand{\ba}{\begin{aligned}}
\newcommand{\ea}{\end{aligned}}
\newcommand{\lb}{\label}

\begin{document}

\title[Return to equilibrium for collisionless gas]{On the speed of approach to equilibrium\\ for a collisionless gas}

\author[K. Aoki]{Kazuo Aoki}
\address[K.A.]{Kyoto University, Department of Mechanical Engineering and Science, 
Graduate School of Engineering, Kyoto 606-8501, Japan}
\email{aoki@aero.mbox.media.kyoto-u.ac.jp}

\author[F. Golse]{Fran\c cois Golse}
\address[F. G.]{Ecole polytechnique, Centre de Math\'ematiques Laurent Schwartz, 
91128 Palaiseau Cedex, France}
\email{golse@math.polytechnique.fr}

\begin{abstract}
We investigate the speed of approach to Maxwellian equilibrium for a collisionless gas enclosed in a vessel whose wall are kept at a uniform, 
constant temperature, assuming diffuse reflection of gas molecules on the vessel wall. We establish lower bounds for potential decay rates
assuming uniform $L^p$ bounds on the initial distribution function. We also obtain a decay estimate in the spherically symmetric case. We 
discuss with particular care the influence of low-speed particles on thermalization by the wall.
\end{abstract}




\maketitle

\rightline{\textit{In memory of Carlo Cercignani (1939-2010)}}

\section{Introduction}

Relaxation to equilibrium is a fundamental issue in statistical mechanics. In the kinetic theory of gases, convergence to equilibrium is formulated 
in terms of the asymptotic behavior of solutions of the Boltzmann equation in the long time limit. In \cite{Cerci82}, Cercignani formulated a quite
remarkable conjecture, asserting the domination of Boltzmann's relative H functional by a constant multiple of Boltzmann's entropy production
functional. If true, Cercignani's conjecture would imply exponential relaxation to equilibrium for space homogeneous solutions of the Boltzmann 
equation, and would have important consequences even for space inhomogeneous solutions. This conjecture was subsequently disproved in
a series of papers by Bobylev, Cercignani and Wennberg \cite{Boby88, BobyCerci99, Wennb97}. And yet, in the words of Villani, ``Cercignani's
conjecture is sometimes true, and almost always true'' \cite{Villani03}. Slightly later, Desvillettes and Villani \cite{DesVil05} studied the speed of
convergence to equilibrium for space inhomogeneous solutions of the Boltzmann equation, under rather stringent conditions of regularity, decay
at large velocities and positivity, but without assuming that the initial data is close to equilbrium. They found that the distance between the distribution 
function at time $t$ and the final equilibrium state decays like $O(t^{-m})$ for each $m>0$ as $t\to+\infty$. Their result makes critical use of the 
Boltzmann collision integral and entropy production, and of their combination with the free transport operator (the left-hand side of the Boltzmann 
equation.) In an earlier paper \cite{Desvil89}, Desvillettes had proved that Boltzmann's entropy production can be regarded as a distance between 
the distribution function and the set of Maxwellian states. On the other hand, Guo proved in \cite{Guo09} that, for initial data close enough to a 
uniform Maxwellian, solutions of the Boltzmann equation converge to equilibrium at exponential speed.

In \cite{DesVil05}, Desvillettes and Villani treated the cases where the spatial domain is the flat torus (equivalently, the Euclidian space with 
periodic initial data), or a smooth bounded domain with either the specular or bounce-back reflection conditions at the boundary. Therefore,
convergence to equilibrium is driven by the collision process alone, and the role of the boundary condition is rather indirect. In the long time
limit, the solution of the Boltzmann equation converges to a uniform Maxwellian state whose density and temperature are given in terms of 
the initial data by the conservation of total mass and total energy. The boundary condition is involved only through its compatibility with the
conservation laws of total mass and energy. In a more recent work \cite{Villani09}, Villani extended their result to the case where the gas is 
enclosed in a vessel whose walls are maintained at a constant temperature, modeled with an accommodation or diffuse reflection condition.
In that case, mass is the only quantity that is conserved under the dynamics of the Boltzmann equation, so that the final equilibrium state is
the centered, uniform Maxwellian with the wall temperature and same total mass as the initial data. In that case, the boundary condition has
a direct influence on the relaxation to equilibrium. Nevertheless, the combined effect of entropy production by the collision process and 
of free streaming, an example of what is termed ``hypocoercivity'', is the driving mechanism used by Villani in his proof of relaxation to 
equilibrium \cite{Villani09}. 

However, it seems obvious from physical consideration that the accommodation or diffuse reflection condition on a wall maintained at some
constant temperature is enough to drive the gas to an equilibrium state, and that the role of the collision process is not as important as in the
periodic, or specular or bounce-back reflection cases --- which are in any case less realistic. 

In order to establish the role of Boltzmann's collision integral in the rate of relaxation to equilibrium, we consider the case of a collisionless gas 
enclosed in a vessel whose surface is maintained at a constant temperature, and investigate the speed of approach to equilibrium for such 
a system. The present paper provides a theoretical confirmation of the numerical results obtained by the authors in \cite{TsujiAokiFG2010}
in collaboration with Tsuji. Basically,  while the collision process is not necessary in the convergence to equilibrium, it greatly influences the 
speed of relaxation to equilibrium. In the same spirit, Desvillettes and Salvarani \cite{DesvilSalvarani} have investigated the speed of relaxation 
to equilibrium in the case of linear collisional models where the collision frequency is not uniformly bounded away from $0$, by adapting 
the hypocoercivity method.

We dedicate this piece of work to the memory of our teacher, friend and colleague Carlo Cercignani, in recognition of his leading role in the 
mathematical theory of the Boltzmann equation.

\section{Statement of the problem}

Consider a collisionless gas enclosed in a container materialized by a bounded domain $\Om$ in $\bR^N$ whose boundary $\d\Om$ is maintained 
at a constant temperature $\th_w>0$. Its distribution function $F\equiv F(t,x,v)$ satisfies
\be\lb{BVPIdealGas}
\left\{
\ba
{}&\d_tF+v\cdot\grad_xF=0\,,\quad\quad\quad\quad\quad\quad\quad\quad\quad\quad\quad\quad\quad (x,v)\in\Om\times\bR^N,
\\
&F(t,x,v)|v\!\cdot\! n_x|=\int_{v'\!\cdot n_x>0}\!\!F(t,x,v')v'\!\cdot\! n_xK(x,v,dv')\,, \quad(x,v)\in\Ga_-\,,
\\
&F\rstr_{t=0}=F^{in}\,.
\ea
\right.
\ee

The notation $n_x$ designates the unit outward normal at the point $x\in\d\Om$, and we have used the notation
$$
\ba
\Ga_+:=\{(x,v)\in\d\Om\times\bR^N\,|\,v\cdot n_x>0\}\,,
\\
\Ga_-:=\{(x,v)\in\d\Om\times\bR^N\,|\,v\cdot n_x<0\}\,.
\ea
$$
The boundary condition is given by the measure-valued kernel $K(x,v,dv')\ge 0$ satisfying the assumptions

\smallskip
\noindent
(i) for each $x\in\d\Om$, one has
$$
\int_{v\cdot n_x<0}K(x,v,dv')dv=dv'
$$
(ii) for each $x\in\d\Om$, one has
$$
\cM_{(1,0,\th)}(v)|v\!\cdot\! n_x|=\int_{v'\!\cdot n_x>0}\cM_{(1,0,\th)}(v')v'\!\cdot\! n_xK(x,v,dv')\,,\quad v\!\cdot\! n_x<0\,,
$$
iff $\th=\th_w$. Throughout this paper, we use the following notation for the Maxwellian distribution:
$$
\cM_{(\rho,u,\th)}(v)=\frac{\rho}{(2\pi\th)^{N/2}}e^{-\frac{|v-u|^2}{2\th}}
$$
(Henceforth, we denote for simplicity $M_w=\cM_{(1,0,\th_w)}$.) With these notations, the case of diffuse reflection corresponds with
$$
K(x,v,dv'):=\frac{\cM_{(1,0,\th_w)}(v)|v\cdot n_x|dv'}{\displaystyle\int_{u\cdot n_x<0}\cM_{(1,0,\th_w)}(u)|u\cdot n_x|du}\,.
$$
The case of an accommodation boundary condition (i.e. the so-called Maxwell-type condition) corresponds with
$$
\ba
K(x,v,dv')&:=(1-\a(x))\de(v'-v+2(v\cdot n_x)n_x)
\\
&+\a(x)\frac{\cM_{(1,0,\th_w)}(v)|v\cdot n_x|dv'}{\displaystyle\int_{u\cdot n_x<0}\cM_{(1,0,\th_w)}(u)|u\cdot n_x|du}\,.
\ea
$$

In the case of diffuse reflection, Arkeryd and Nouri \cite{ArkeNouri97} have proved that the solution of (\ref{BVPIdealGas}) satisfies 
$$
F(t+\cdot,\cdot,\cdot)\to\frac{\cM_{(1,0,\th_w)}}{|\Om|}\iint_{\Om\times\bR^N} F^{in}dxdv\hbox{ in }L^1([0,T]\times\Om\times\bR^N)
$$
as $t\to+\infty$.

In the present paper, we study in detail the following

\smallskip
\noindent
\textbf{Problem:}  Is there a decay rate $E(t)\to 0$ as $t\to +\infty$, for which an asymptotic estimate of the form
\be\lb{DecayUnifL1}
\ba
\left\|F(t+\cdot,\cdot,\cdot)-\frac{\cM_{(1,0,\th_w)}}{|\Om|}\iint_{\Om\times\bR^N} F^{in}dxdv\right\|_{L^1([0,T]\times\Om\times\bR^N)}=O(E(t))
\ea
\ee
holds for each solution $F$ of (\ref{BVPIdealGas}) uniformly as $F^{in}$ runs through certain classes of bounded subsets of $L^1(\bR^N\times\bR^N)$?

\section{Negative results and lower bounds on $E$}\lb{S-NegResult}

We first consider the problem of finding a decay rate that is uniform as $F^{in}$ runs through all bounded subsets of  $L^1(\bR^N\times\bR^N)$.

\begin{Prop}\lb{P-NonDecayL1}
There does not exist any decay rate function $E:\,\bR_+\to\bR_+$ such that $E(t)\to 0$ as $t\to+\infty$ and
\be\lb{DecayL1}
\ba
\left\|F(t+\cdot,\cdot,\cdot)-\frac{\cM_{(1,0,\th_w)}}{|\Om|}\iint_{\Om\times\bR^N} F^{in}dxdv\right\|_{L^1([0,T]\times\Om\times\bR^N)}
\\
\le
E(t)\|F^{in}\|_{L^1(\bR^N\times\bR^N)}
\ea
\ee
for each $t\ge 0$.
\end{Prop}

Since the initial boundary value problem (\ref{BVPIdealGas}) is linear, if there is a long time decay as in (\ref{DecayUnifL1}) uniformly as the initial 
data $F^{in}$ runs through bounded subsets of $L^1$, an inequality of the form (\ref{DecayL1}) must hold.

\begin{proof}
Assume that the origin $0\in\Om$ --- if not, one can apply a translation in the $x$-variable without changing the nature of the problem --- and set 
$R=\tfrac12\Dist(0,\d\Om)$. We henceforth denote by $\bB$ the closed unit  ball of $\bR^N$, and by $\indc_A$ the indicator function of the set $A$.

For $0<\eps\ll 1$, set 
$$
F^{in}(x,v)=\frac1{\eps^{2N}}\indc_{\eps\bB}(x)\indc_{\eps\bB}(v)
$$
(with $\eps\bB:=\{\eps z\,|\,z\in\bB\}$) so that
$$
\iint F^{in}(x,v)dxdv=|\bB|^2\,.
$$
Call $\tau(x,v):=\inf\{t\ge 0\,|\,x-tv\notin\Om\}$; obviously the solution $F$ of (\ref{BVPIdealGas}) satisfies $F\ge \Phi$, where $\Phi$ is the solution 
of the same free transport equation, with the same initial data, but with absorbing boundary condition:
\be\lb{AbsorbIdealGas}
\left\{
\ba
{}&\d_t\Phi+v\cdot\grad_x\Phi=0\,, \quad&&(x,v)\in\Om\times\bR^N,
\\
&\Phi(t,x,v)=0\,, &&(x,v)\in\Ga_-\,,
\\
&\Phi\rstr_{t=0}=F^{in}\,.&&{}
\ea
\right.
\ee
Solving for $\Phi$ along characteristics, one finds that
$$
\Phi(t,x,v)=F^{in}(x-tv,v)\indc_{0\le t\le\tau(x,v)}\,.
$$
Now
$$
\ba
\left\|F(t+\cdot,\cdot,\cdot)-\frac{\cM_{(1,0,\th_w)}}{|\Om|}\iint_{\Om\times\bR^N} F^{in}dxdv\right\|_{L^1([0,T]\times\Om\times\bR^N)}
\\
=
\left\|F(t+\cdot,\cdot,\cdot)-\frac{|\bB|^2}{|\Om|}\cM_{(1,0,\th_w)}\right\|_{L^1([0,T]\times\Om\times\bR^N)}
\\
\ge
\int_0^T\int_\Om\int_{\bR^N}\left(F(t+s,x,v)-\frac{|\bB|^2}{|\Om|}\cM_{(1,0,\th_w)}(v)\right)^+dvdxds\,.
\ea
$$
Henceforth, if $A\subset\bR^3$ and $u\in\bR^3$, we denote $A+u:=\{z+u\,|\,z\in A\}$. Since the function $z\mapsto z^+=\max(z,0)$ is 
nondecreasing, one has
$$
\ba
\left\|F(t+\cdot,\cdot,\cdot)-\frac{\cM_{(1,0,\th_w)}}{|\Om|}\iint_{\Om\times\bR^N} F^{in}dxdv\right\|_{L^1([0,T]\times\Om\times\bR^N)}
\\
\ge\int_0^T\int_\Om\int_{\bR^N}\left(\Phi(t+s,x,v)-\frac{|\bB|^2}{|\Om|}\cM_{(1,0,\th_w)}(v)\right)^+dvdxds
\\
=\int_0^T\int_\Om\int_{\bR^N}\left(\frac1{\eps^{2N}}\indc_{\eps\bB+(t+s)v}(x)\indc_{\eps\bB}(v)\indc_{0\le t+s\le\tau(x,v)}\right.\quad\quad
\\
\left.-\frac{|\bB|^2}{|\Om|}\cM_{(1,0,\th_w)}(v)\right)^+dvdxds
\\
\ge\int_0^T\int_\Om\int_{\bR^N}\left(\frac1{\eps^{2N}}\indc_{\eps\bB+(t+s)v}(x)\indc_{\eps\bB}(v)\indc_{0\le(t+s)|v|\le R-\eps}\right.\quad
\\
\left.-\frac{|\bB|^2}{|\Om|}\cM_{(1,0,\th_w)}(v)\right)^+dvdxds
\ea
$$
since $\tau(x,v)|v|\ge 2R-\eps$ for each $x\in\eps\bB$. Now
$$
\ba
\left(\frac1{\eps^{2N}}\indc_{\eps\bB+(t+s)v}(x)\indc_{\eps\bB}(v)\indc_{0\le(t+s)|v|\le R-\eps}
	-\frac{|\bB|^2}{|\Om|}\cM_{(1,0,\th_w)}(v)\right)^+
\\
=
\frac1{\eps^{2N}}\indc_{\eps\bB+(t+s)v}(x)\indc_{\eps\bB}(v)\indc_{0\le(t+s)|v|\le R-\eps}
	\left(1-\frac{\eps^{2N}|\bB|^2}{|\Om|}\cM_{(1,0,\th_w)}(v)\right)^+
\\
\ge
\frac1{\eps^{2N}}\indc_{\eps\bB+(t+s)v}(x)\indc_{\eps\bB}(v)\indc_{0\le(t+s)|v|\le R-\eps}
	\left(1-\frac{\eps^{2N}|\bB|^2}{|\Om|(2\pi\th_w)^{N/2}}\right)^+\,.
\ea
$$
Then
$$
\ba
\int_0^T\int_\Om\int_{\bR^N}\frac1{\eps^{2N}}\indc_{\eps\bB+(t+s)v}(x)\indc_{\eps\bB}(v)\indc_{0\le(t+s)|v|\le R-\eps}dvdxds
\\
=
|\bB|\int_0^T\int_{\bR^N}\frac1{\eps^N}\indc_{\eps\bB}(v)\indc_{0\le(t+s)|v|\le R-\eps}dvds
\\
=
|\bB|\int_0^T\int_{\bR^N}\frac1{\eps^N}\indc_{0\le|v|\le\min(\eps,(R-\eps)/t+s)}dvds
\\
=|\bB|\int_0^T\frac1{\eps^N}\min\left(\eps^N,\frac{(R-\eps)^N}{(t+s)^N}\right)ds
\\
\ge|\bB|\frac{T}{\eps^N}\min\left(\eps^N,\frac{(R-\eps)^N}{(t+T)^N}\right)\,.
\ea
$$

Therefore, if there is a uniform decay rate $E$ as in (\ref{DecayL1}), it must satisfy
$$
T|\bB|\left(1-\frac{\eps^{2N}|\bB|^2}{|\Om|(2\pi\th_w)^{N/2}}\right)^+\left(\min\left(1,\frac{(R-\eps)}{\eps(t+T)}\right)\right)^N\le |\bB|^2E(t)\,.
$$

Here $\eps$ only has to satisfy $0<\eps<R$.  If one lets $\eps\to 0^+$, this contradicts the existence of a profile $E$ in (\ref{DecayL1}) 
such that $E(t)\to 0$ as $t\to+\infty$. 
\end{proof}

\smallskip
The negative result in Proposition \ref{P-NonDecayL1} could have been easily anticipated: as $\eps\to 0$, our choice of $F^{in}$ converges 
to $|\bB|^2\de_{x=0}\de_{v=0}$ in the weak topology of measures. Then particles distributed under $\de_{x=0}\de_{v=0}$ do not move, hence
cannot reach the boundary and therefore cannot be thermalized. Thus, in some sense, the problem of finding a decay rate as in (\ref{DecayL1}) 
is trivial if stated in this way.

\smallskip
Before going further, let us comment on the kind of inequality considered in Proposition \ref{P-NonDecayL1}. In order to measure the speed of
approach to equilibrium for the problem (\ref{BVPIdealGas}), it might be more natural to consider a pointwise in time inequality of the type
\be\lb{PtwseDecay}
\left\|F(t,\cdot,\cdot)-\frac{\cM_{(1,0,\th_w)}}{|\Om|}\iint_{\Om\times\bR^N} F^{in}dxdv\right\|_{L^1(\Om\times\bR^N)}
\\
\le
E(t)\|F^{in}\|_{L^1(\bR^N\times\bR^N)}
\ee
instead of (\ref{DecayL1}). Observe that the l.h.s. of (\ref{DecayL1}) is the time-average of the l.h.s. of (\ref{PtwseDecay}), so that (\ref{PtwseDecay})
is a stronger statement than (\ref{DecayL1}) --- as a matter of fact, (\ref{PtwseDecay}) entails  (\ref{DecayL1}) with $E(t)$ changed into $TE(t)$.
Therefore, negative results obtained on time-averaged estimates like (\ref{DecayL1}) are stronger than their analogues for the corresponding
pointwise estimate similar to (\ref{PtwseDecay}). This is why we sought to confirm numerically pointwise in time bounds like (\ref{PtwseDecay}) in
our previous paper \cite{TsujiAokiFG2010}, and not time-averaged estimates as in the work of Arkeryd-Nouri \cite{ArkeNouri97}. Throughout the
present section, we discuss negative results on the speed of approach to equilibrium and therefore consider time-averaged decay estimates.

\smallskip
A slightly less trivial issue is to investigate the existence of such a decay rate under the additional assumption ruling out the possibility of a 
complete concentration as above. This could be done by controling not only the $L^1$ norm of the initial data, but (for instance) its entropy, 
or relative entropy with respect to the final state. For instance, one could look instead at a bound of the form
\be\lb{DecayH}
\ba
\left\|F(t+\cdot,\cdot,\cdot)-\frac{\cM_{(1,0,\th_w)}}{|\Om|}\iint_{\Om\times\bR^N} F^{in}dxdv\right\|_{L^1([0,T]\times\Om\times\bR^N)}
\\
\le E(t)\iint_{\Om\times\bR^N}F^{in}(x,v)|\ln F^{in}(x,v)|dxdv\,.
\ea
\ee

\begin{Prop}\lb{P-DecayLlnL}
If there is a time decay rate $E$ such that the solution of (\ref{BVPIdealGas}) satisfies (\ref{DecayH}), then it must satisfy
$$
E(t)\ge\frac{C}{\ln t}\quad\hbox{ as }t\to+\infty\,.
$$
\end{Prop}

\begin{proof}
For the same initial data as before
$$
\ba
\iint F^{in}(x,v)|\ln F^{in}(x,v)|dxdv&=\iint \frac1{\eps^{2N}}\indc_{\eps B}(x)\indc_{\eps B}(v)\ln\left(\frac1{\eps^{2N}}\right)dxdv
\\
&=2N|\bB|^2|\ln\eps|\,.
\ea
$$

Proceeding as before to obtain a lower bound for the l.h.s. of (\ref{DecayH}), one arrives at the inequality
$$
T|\bB|\left(1-\frac{\eps^{2N}|\bB|^2}{|\Om|(2\pi\th_w)^{N/2}}\right)^+\left(\min\left(1,\frac{(R-\eps)}{\eps(t+T)}\right)\right)^N
	\le 2N|\bB|^2|\ln\eps|E(t)\,.
$$
Choosing $\eps_0>0$ small enough, this inequality takes the form
$$
C\left(\min\left(1,\frac{1}{\eps t}\right)\right)^N\le|\ln\eps|E(t)
$$
for all $t>T$ and all $\eps\in(0,\eps_0)$. Picking then $\eps=1/t$ with $t\gg 1$ so that in particular $t>\max (T,1/\eps_0)$ leads to 
$$
E(t)\ge\frac{C}{\ln t}\quad\hbox{ as }t\to+\infty\,.
$$
\end{proof}

\smallskip
The constraint on the decay rate $E$ obtained in Proposition \ref{P-DecayLlnL} remains the same if one replaces the r.h.s. in the bound 
(\ref{DecayH}) with either
$$
E(t)\iint_{\Om\times\bR^N}F^{in}(x,v)(1+|v|^2+|\ln F^{in}(x,v)|)dxdv
$$
or
$$
E(t)\iint_{\Om\times\bR^N}\left(F^{in}(x,v)\ln\left(\frac{F^{in}}{\bar\rho\cM_{(1,0,\th_w)}}\right)-F^{in}+\bar\rho\cM_{(1,0,\th_w)}\right)(x,v)dxdv\,,
$$
where
$$
\bar\rho=\frac1{|\Om|}\iint_{\Om\times\bR^N} F^{in}dxdv\,.
$$
Indeed, compared to the r.h.s. in the bound (\ref{DecayH}), both these quantities involve in addition the total energy integral
$$
\iint_{\Om\times\bR^N}|v|^2F^{in}(x,v)dxdv\,.
$$
Since the inequality (\ref{DecayH}) is tested with velocity distribution functions concentrated near $v=0$, this additional term is negligible 
and its presence does not change the final result.

\smallskip
Of course, prescribing an entropy bound is not the only possibility for avoiding concentrations in the initial data $F^{in}$. For instance, one 
could ask for a control of $F^{in}$ in some $L^p$-norm with $1<p\le\infty$. In other words, one could investigate the possibility of an inequality 
of the form
\be\lb{DecayLp}
\ba
\left\|F(t+\cdot,\cdot,\cdot)-\frac{\cM_{(1,0,\th_w)}}{|\Om|}\iint_{\Om\times\bR^N} F^{in}dxdv\right\|_{L^1([0,T]\times\Om\times\bR^N)}
\\
\le E(t)\|F^{in}\|_{L^p(\Om\times\bR^N)}\,.
\ea
\ee

\begin{Prop}\lb{P-DecayLp}
If there is a time decay rate $E$ such that the solution of (\ref{BVPIdealGas}) satisfies (\ref{DecayLp}), then it must satisfy
$$
E(t)\ge\frac{C}{t^{N\min(1,2/p')}}\quad\hbox{ as }t\to+\infty\,.
$$
\end{Prop}

\begin{proof}
With the same initial data as before
$$
\|F^{in}\|_{L^p(\Om\times\bR^N)}=\frac{|\bB|^{2/p}}{\eps^{2N/p'}}\,,
$$
with $p'=\frac{p}{p-1}<\infty$ the dual exponent of $p$, so that, estimating the l.h.s. by the same argument as before, one arrives at the inequality
$$
T|\bB|\left(1-\frac{\eps^{2N}|\bB|^2}{|\Om|(2\pi\th_w)^{N/2}}\right)^+\left(\min\left(1,\frac{(R-\eps)}{\eps(t+T)}\right)\right)^N
\le\frac{|\bB|^{2/p}}{\eps^{2N/p'}}E(t)\,.
$$
Assume that $p\ge 2$. Since the growth of the l.h.s. in $1/\eps$ is smaller than that of the r.h.s., there is no point in letting $\eps\to 0$. One therefore 
chooses $\eps>0$ small enough but fixed, so that $\eps<R/2$ and $\eps^{2N}|\bB|^2<\tfrac12|\Om|(2\pi\th_w)^{N/2}$, so that 
$$
\frac{T}2|\bB|\left(\min\left(1,\frac{R}{2\eps(t+T)}\right)\right)^N\le\frac{|\bB|^{2/p}}{\eps^{2N/p'}}E(t)\,.
$$
In that case, one sees that there exists $C=C(T,R,\eps)>0$ such that
$$
E(t)\ge\frac{C}{t^N}\quad\hbox{ as }t\gg 1\,.
$$

If $1<p<2$, by picking $\eps_0>0$ small enough, the inequality above takes the form
$$
C\left(\min\left(1,\frac{1}{\eps t}\right)\right)^N\le\frac{|\bB|^{2/p}}{\eps^{2N/p'}}E(t)\,,
$$
for all $t>0$ and all $\eps\in(0,\eps_0)$. Picking $\eps=1/t$ with $t\gg 1$ so that in particular $t>\max(T,1/\eps_0)$, one concludes from this last 
estimate that
$$
E(t)\ge \frac{C}{t^{2N/p'}}\quad\hbox{ as }t\to+\infty\,.
$$
\end{proof}

\smallskip
At this point, it is worth comparing the result in Proposition \ref{P-DecayLp} for the collisionless gas driven to equilbrium by the boundary, and the 
following important result, obtained by Ukai, Point and Ghidouche \cite{UkaiPointGhid78}, for the Boltzmann equation linearized at a uniform 
Maxwellian equilibrium.

Consider the Boltzmann equation linearized at the uniform Maxwellian state $M=\cM_{(1,0,1)}$, assuming hard sphere collisions,
in the case where the spatial domain is the periodic box with unit side $\bT^3$. In other words, consider the Cauchy problem
\be\lb{CauchyPbLinBoltz}
\left\{
\ba
{}&(\d_t+v\cdot\grad_x)g+\cL_Mg=0\,,\quad (t,x,v)\in\bR_+^*\times\bT^3\times\bR^3\,,
\\
&g\rstr_{t=0}=g^{in}
\ea
\right.
\ee
The linearized collision operator is given by
$$
\cL_M\phi(v):=\iint_{\bR^3\times\bS^2}(\phi(v')+\phi(v'_*)-\phi(v)-\phi(v_*))|(v-v_*)\cdot\om|M(v_*)dv_*d\om
$$
with
$$
\left\{
\ba
{}&v'=v-(v-v_*)\cdot\om\om
\\
&v'_*=v_*+(v-v_*)\cdot\om\om
\ea
\right.
$$
where $\om$ runs through the unit sphere $\bS^2$. It is understood that the notation $\cL_M\phi(t,x,v)$ designates $(\cL_M\phi(t,x,\cdot))(v)$. 
Define
$$
\Pi g^{in}=\la g^{in}\ra+\la vg^{in}\ra\cdot v+\la(\tfrac13|v|^2-1)g^{in}\ra\tfrac12(|v|^2-3)
$$
with the notation 
$$
\la\phi\ra=\int_{\bR^3}\phi(v)M(v)dv\,.
$$

\begin{Thm}\cite{UkaiPointGhid78}
There exist $\g_*>0$, and, for each $\g\in]0,\g_*[$, a positive constant $C_\g$ such that, for each $g^{in}\in L^2(\bT^3\times\bR^3;Mdvdx)$,
the solution $g$ of (\ref{CauchyPbLinBoltz}) satisfies the inequality
$$
\left\|g(t,\cdot,\cdot)-\int_{\bT^3}\Pi g^{in}dx\right\|_{L^2(\bT^3\times\bR^3;Mdvdx)}\le C_\g e^{-\g t}\|g^{in}\|_{L^2(\bT^3\times\bR^3;Mdvdx)}
$$
for each $t\ge 0$.
\end{Thm}

This result expresses that there is a spectral gap for the linearized Boltzmann (unbounded) operator
$$
\bL g(x,v)=v\cdot\grad_xg(x,v)+\cL_Mg(x,v)
$$
on $L^2(\bT^3\times\bR^3;Mdvdx)$ with domain 
$$
D(\bL)=\{\phi\in L^2(\bT^3\times\bR^3;Mdvdx)\,|\,v\cdot\grad_x\phi\hbox{ and }(1+|v|)\phi\in L^2(\bT^3\times\bR^3;Mdvdx)\}\,.
$$
In other words, the spectrum of $\bL$ is included in a domain of the form
$$
\hbox{spec}(\bL)\subset\{0\}\cup\{z\in\bC\,|\,\hbox{Re}(z)\ge\th\}
$$
for some $\th>0$ (the spectral gap). The eigenvalue $0$ corresponds to the five-dimensional nullspace of $\bL$, given by
$$
\Ker\bL=\{a+b\cdot v+c\tfrac12(|v|^2-3)\,|\,a,c\in\bR\hbox{ and }b\in\bR^3\}\,.
$$
One reason explaining the spectral gap as above is the fact that $\cL_M$ is an (unbounded) Fredholm operator in the $v$ variable, of the
form
$$
\cL_M\phi(v)=\nu(|v|)\phi(v)-\cK\phi
$$
where $\nu(|v|)$ is a multiplication operator by the collision frequency satisfying the inequality $\nu_*(1+|v|)\le\nu(|v|)\le\nu^*(1+|v|)$ for some 
positive constants $\nu_*$ and $\nu^*$.

Proposition \ref{P-DecayLp} for $p=2$ rules out the possibility of a similar spectral gap property for the problem (\ref{BVPIdealGas}). 

Indeed, assume for simplicity that $\th_w=1$ in the problem (\ref{BVPIdealGas}). Then by the Cauchy-Schwarz inequality
$$
\ba
|\Om|^{1/2}&\left\|\frac{F(t,\cdot,\cdot)}{M}-\frac1{|\Om|}\iint_{\Om\times\bR^3}F^{in}dxdv\right\|_{L^2(\Om\times\bR^3;Mdvdx)}
\\
&\ge
\left\|F(t,\cdot,\cdot)-\frac{M}{|\Om|}\iint_{\Om\times\bR^3}F^{in}dxdv\right\|_{L^1(\Om\times\bR^3;dvdx)}
\ea
$$
so that a spectral gap inequality of the form
\be\lb{SpecGapIdealGas}
\ba
\left\|\frac{F(t,\cdot,\cdot)}{M}-\frac1{|\Om|}\iint_{\Om\times\bR^3}F^{in}dxdv\right\|_{L^2(\Om\times\bR^3;Mdvdx)}
\\
\le
C_\g e^{-\g t}\left\|\frac{F^{in}}{M}\right\|_{L^2(\Om\times\bR^3;Mdvdx)}
\ea
\ee
would entail
$$
\ba
\left\|F(t+\cdot,\cdot,\cdot)-\frac{M}{|\Om|}\iint_{\Om\times\bR^3}F^{in}dxdv\right\|_{L^1([0,T]\times\Om\times\bR^3;dvdxds)}
\\
\le
T|\Om|^{1/2}C_\g e^{-\g t}\left\|\frac{F^{in}}{M}\right\|_{L^2(\Om\times\bR^3;Mdvdx)}\,.
\ea
$$
Now, for 
$$
F^{in}(x,v)=\frac1{\eps^{2N}}\indc_{\eps\bB}(x)\indc_{\eps\bB}(v)\,,
$$
one has, assuming that $0<\eps<1$,
$$
\left\|\frac{F^{in}}{M}\right\|_{L^2(\Om\times\bR^3;Mdvdx)}
\le
((2\pi)^Ne)^{-1/4}\left\|F^{in}\right\|_{L^2(\Om\times\bR^3;dvdx)}\,.
$$

In other words, the existence of a spectral gap for the problem (\ref{BVPIdealGas}) in $\Om\times\bR^3$ --- i.e. the inequality (\ref{SpecGapIdealGas}) 
would entail the inequality
$$
\ba
\left\|F(t+\cdot,\cdot,\cdot)-\frac{\cM_{(1,0,\th_w)}}{|\Om|}\iint_{\Om\times\bR^3}F^{in}dxdv\right\|_{L^1([0,T]\times\Om\times\bR^3;dvdxds)}
\\
\le
T|\Om|^{1/2}((2\pi)^Ne)^{-1/4}C_\g e^{-\g t}\left\|F^{in}\right\|_{L^2(\Om\times\bR^3;dvdx)}\,,
\ea
$$
which stands in contradiction with Proposition \ref{P-DecayLp}.

\section{Algebraic decay}\lb{S-AlgDecay}

After the negative results obtained in the previous section, we now consider the problem of establishing decay estimates compatible
with the results of section \ref{S-NegResult}.

We shall do so on a special class of boundary value problems, satisfying the following assumptions:

\smallskip
\noindent
(a) the spatial domain $\Om$ is a ball of the Euclidian space $\bR^3$;

\noindent
(b) the distribution function of the collisionless gas at $t=0$ is a radial function of both the space and velocity variables;

\noindent
(c) the boundary condition is a diffuse reflection (complete accommodation), with a constant boundary temperature (no temperature
gradient at the boundary of the vessel.)

\smallskip
None of these assumptions excludes the mechanism producing the negative results in section \ref{S-NegResult}. Indeed, all these
results are based on the choice of an initial distribution function that is concentrated, in the velocity variable, near $v=0$, and in the
space variable, near a point that is away from the boundary. This is obviously possible with the symmetries above, by choosing
the initial data concentrated, in the space variable, near the center of the ball, and in the velocity variable, near $v=0$ as in section
\ref{S-NegResult}. Particles so distributed at time $t=0$ will not reach the boundary of the spatial domain until after very late, and,
meanwhile, the exact nature of the boundary condition is immaterial for such particles.

\smallskip
Therefore, we set the spatial domain to be $\Om=B(0,1)$, the unit ball of $\bR^3$. The temperature on $\d\Om$ is set to be 
$\th_w=1$. Consider the initial boundary value problem
\be\lb{BVPSpher}
\left\{
\ba
{}&(\d_t+v\cdot\grad_x)f=0\,,\quad &&(t,x,v)\in\bR_+\times\Om\times\bR^3\,,
\\
&f(t,x,v)=\frac{\la f\ra_+}{\la M\ra_+}M(v)\,,&&v\cdot n_x<0\,,\,\,x\in\d\Om\,,
\\
&f\rstr_{t=0}=f^{in}(|x|,|v|)\,,
\ea
\right.
\ee
with the notation 
$$
\la\phi\ra_+=\int_{v\cdot n_x>0}\phi(v)v\cdot n_xdv\,,
$$
and we recall the notation $M=\cM_{(1,0,1)}$ for the Maxwellian distribution with density $1$, zero bulk velocity and temperature $1$.  

The main result in this section is the following.

\begin{Thm}\lb{T-AlgDecayBoltz}
Assume that $0\le f^{in}(x,v)\le CM(v)$ a.e. in $(x,v)\in\Om\times\bR^3$. Then the solution of (\ref{BVPSpher}) satisfies
$$
\iint_{\Om\times\bR^3}\left|\frac{f(t,x,v)}{M(v)}-\tfrac1{|\Om|}\iint_{\Om\times\bR^3}f^{in}(|x|,|v|)dxdv\right|^pM(v)dv
=O\left(\frac{1}{t^3}\right)+O\left(\frac{1}{t^p}\right)
$$
so that
$$
\left\|\frac{f(t,\cdot,\cdot)}{M}-\tfrac1{|\Om|}\iint_{\Om\times\bR^3}f^{in}(|x|,|v|)dxdv\right\|_{L^p(\Om\times\bR^3;Mdvdx)}
=O\left(\frac{1}{t^{\min(1,3/p)}}\right)
$$
for all $1\le p<\infty$.
\end{Thm}

The reader might want to compare the upper bound in this theorem with the lower bound in Proposition \ref{P-DecayLp}: this is done by 
taking $p=\infty$ in the estimate (\ref{DecayLp}) and in Proposition \ref{P-DecayLp}, and $p=1$ in the theorem above. The difference in
the exponents observed suggests that one of these results at least is not sharp. The numerical investigations in \cite{TsujiAokiFG2010}
suggests that it is the lower estimate in Proposition \ref{P-DecayLp} that is sharp, while the upper bound in the theorem above is not.

\smallskip
Our main task in the proof of this theorem is to use the symmetries in the problem (\ref{BVPSpher}) to reduce it to a renewal integral 
equation for a scalar unknown quantity, and to use classical results for such equations. 

The role of symmetries in this problem is summarized in the following statement.

\begin{Lem}\lb{L-BoundarySym}
If $f$ is the solution of the initial boundary value problem (\ref{BVPSpher}), its outgoing flux at the boundary
$$
\mu:=\int_{v\cdot n_x>0}f(t,x,v)v\cdot n_xdv
$$
is independent of the position variable $x$ as $x$ runs through $\d\Om$, and a function of $t\ge 0$ only.
\end{Lem}

\begin{proof}
If $R\!\in\! O_3(\bR)$ is any orthogonal matrix, the function $f_R(t,x,v)\!:=\!f(t,Rx,Rv)$ obviously satisfies $f_R\rstr_{t=0}=f\rstr_{t=0}$,
and $\la f_R\ra_+(t,x)=\la f\ra_+(t,Rx)$, while
$$
\ba
(\d_t+v\cdot\grad_x)f_R(t,x,v)&=\d_tf(t,Rx,Rv)+v\cdot(R^T\grad_xf)(t,Rx,Rv)
\\
&=\d_tf(t,Rx,Rv)+Rv\cdot\grad_xf(t,Rx,Rv)=0\,.
\ea
$$
By uniqueness of the solution of (\ref{BVPSpher}), we conclude that $f_R=f$ for each $R\in O_3(\bR)$. Thus
$$
\la f\ra_+(t,x)=\la f_R\ra_+(t,x)=\la f\ra_+(t,Rx)
$$
for each $(t,x)\in\bR_+\times\d\Om$ and $R\in O_3(\bR)$: hence $\mu:=\la f\ra_+$ is independent of $x$ as $x$ runs through 
$\d\Om$.
\end{proof}

\smallskip
\begin{proof}[Proof of Theorem \ref{T-AlgDecayBoltz}]

With Lemma \ref{L-BoundarySym}, we easily arrive at an integral equation satisfied by $\mu$. Pick $x\in\d\Om$, and set 
$\th=\widehat{(n_x,v)}$; the backward exit time 
$$
\tau_{x,v}:=\inf\{t>0\,|\,x-tv\notin\overline{\Om}\}
$$
is given by the explicit formula
$$
\tau_{x,v}=\frac{2\cos\th}{|v|}\,,
$$
as can be seen on Figure 1.

\begin{figure}

\includegraphics[width=10cm]{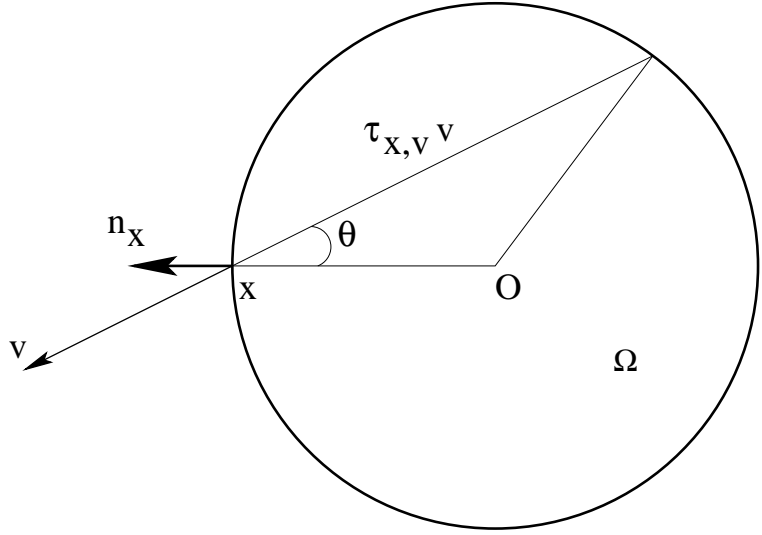}

\caption{The exit time for a spherical domain}

\end{figure}

By the method of characteristics, for each $x\in\d\Om$ and $v\in\bR^3$ such that $v\cdot n_x>0$, one has
$$
f(t,x,v)=\indc_{t>\tau_{x,v}}\frac{M(v)}{\la M\ra_+}\mu(t-\tau_{x,v})+\indc_{t<\tau_{x,v}}f^{in}(|x-tv|,|v|)\,.
$$
Hence
$$
\ba
\mu(t)&=\int_{v\cdot n_x>0}f(t,x,v)v\cdot n_xdv
\\
&=\frac1{\la M\ra_+}\int_{v\cdot n_x>0}\mu(t-\tau_{x,v})\indc_{t>\tau_{x,v}}M(v)v\cdot n_xdv
\\
&+
\int_{v\cdot n_x>0}f^{in}(|x-tv|,|v|)\indc_{t<\tau_{x,v}}v\cdot n_xdv
\ea
$$
Introducing spherical coordinates in $v$, with $n_x$ being the polar direction and $\th$ the colatitude, we see on Figure 1 that, 
for each $x\in\d\Om$, one has $\tau_{x,v}|v|=2\cos\th$, so that
$$
\ba
\int_{v\cdot n_x>0}&\mu(t-\tau_{x,v})\indc_{t>\tau_{x,v}}M(v)v\cdot n_xdv
\\
&=
2\pi\int_0^\infty\int_0^{\pi/2}\mu\left(t-\frac{2\cos\th}{|v|}\right)\indc_{t|v|>2\cos\th}M(|v|)|v|\cos\th|v|^2\sin\th d\th d|v|\,,
\ea
$$
abusing the notation $M(|v|)$ to designate the radial function $v\mapsto M(v)$. 

Substituting $y=\cos\th$, this integral is expressed as
$$
\ba
\int_0^\infty\int_0^1&\mu\left(t-\frac{2y}{|v|}\right)\indc_{t|v|>2y}M(|v|)|v|^3ydyd|v|
\\
&=\tfrac14\int_0^\infty\int_0^{2/|v|}\mu\left(t-\tau\right)\indc_{t>\tau}M(|v|)|v|^5\tau d\tau d|v|
\\
&=\tfrac14\int_0^\infty\int_0^t\mu\left(t-\tau\right)\indc_{\tau<2/|v|}M(|v|)|v|^5\tau d\tau d|v|
\\
&=\tfrac14\int_0^t\mu\left(t-\tau\right)\left(\int_0^{2/\tau}M(|v|)|v|^5d|v|\right)\tau d\tau\,.
\ea
$$
We henceforth define
$$
K(\tau):=\frac{\pi\tau}{2\la M\ra_+}\int_0^{2/\tau}M(|v|)|v|^5d|v|\,.
$$
The function $K$ is continuous on $\bR_+^*$, satisfies $K(\tau)\ge 0$ for each $\tau\ge 0$, and
$$
K(\tau)\to 0\hbox{Ê as }\tau\to 0^+\,,\quad K(\tau)\sim\frac{8}{3\tau^5}\hbox{ as }\tau\to+\infty\,.
$$
Therefore $\mu$ satisfies
\be\lb{Renewal}
\mu(t)=\int_0^tK(\tau)\mu(t-\tau)d\tau+S(t)
\ee
where  the source term is
$$
S(t):=\int_{v\cdot n_x>0}f^{in}(|x-tv|,|v|)\indc_{t<\tau_{x,v}}v\cdot n_xdv\,.
$$
This is a renewal equation, to which we shall apply the results of \cite{Feller41}.

Next we consider the source term. Define 
$$
\Phi(|v|):=\Supess_{x\in\Om}f^{in}(|x|,|v|)\,;
$$
assuming that $0\le f^{in}\in L^\infty(\Om\times\bR^3)$ implies that $\Phi\in L^\infty(\bR_+)$. Then, 
$$
\ba
S(t)&\le
\int_{v\cdot n_x>0}\Phi(|v|)\indc_{t<\tau_{x,v}}v\cdot n_xdv
\\
&=
2\pi\int_0^\infty\Phi(|v|)|v|^3\left(\int_0^{\pi/2}\indc_{t|v|<2\cos\th}\cos\th\sin\th d\th\right)d|v|
\\
&=
2\pi\int_0^\infty\Phi(|v|)|v|^3\left(\int_0^1\indc_{t|v|<2y}ydy\right)d|v|
\\
&=
\pi\int_0^\infty\Phi(|v|)|v|^3(1-\tfrac14t^2|v|^2)_+d|v|
\\
&\le
\pi\|\Phi\|_{L^\infty}\int_0^{2/t}|v|^3d|v|=\frac{4\pi\|\Phi\|_{L^\infty}}{t^4}\,.
\ea
$$

Apply Theorem 4 of \cite{Feller41}: with $n=3$, one has
$$
\int_0^\infty t^mK(t)dt<+\infty\quad\hbox{ for }m=0,1,2,3
$$
while
$$
S(t)=o\left(\frac1t\right)\,,\quad\int_t^\infty S(\tau)d\tau=o\left(\frac1t\right)
$$
as $t\to+\infty$. Then 
$$
\mu(t)\to\mu_\infty=\frac{\displaystyle\int_0^\infty S(\tau)d\tau}{\displaystyle\int_0^\infty\tau K(\tau)d\tau}
$$
as $t\to+\infty$, and
$$
\mu(t)-\mu_\infty=o\left(\frac1t\right)\,.
$$

Now let $g:=f/M$, where $f$ is the solution of (\ref{BVPSpher}); $g$ satisfies
\be\lb{BVPSpher/M}
\left\{
\ba
{}&(\d_t+v\cdot\grad_x)g=0\,,\quad &&(t,x,v)\in\bR_+\times\Om\times\bR^3\,,
\\
&g(t,x,v)=\sqrt{2\pi}\mu(t)\,,&&v\cdot n_x<0\,,\,\,x\in\d\Om\,,
\\
&g\rstr_{t=0}=f^{in}(|x|,|v|)/M(v)\,.
\ea
\right.
\ee
The function $g$ is given by
$$
g(t,x,v)=\sqrt{2\pi}\mu(t-\tau_{x,v})\indc_{t>\tau_{x,v}}+\frac{f^{in}(|x|,|v|)}{M(v)}\indc_{t<\tau_{x,v}}\,,
$$
so that
$$
\ba
g(t,x,v)-\sqrt{2\pi}\mu_\infty&=\sqrt{2\pi}\left(\mu(t-\tau_{x,v})-\mu_\infty\right)\indc_{t>\tau_{x,v}}
\\
&+\left(\frac{f^{in}(|x|,|v|)}{M(v)}-\sqrt{2\pi}\mu_\infty\right)\indc_{t<\tau_{x,v}}=I+II\,.
\ea
$$

Now $\tau_{x,v}\le 2/|v|$, so that 
$$
\int_{\bR^3}\indc_{t<\tau_{x,v}}M(v)dv\le\int_{|v|\le 2/t}M(v)dv\sim\tfrac{2^3|\bB^3|}{(2\pi)^{3/2}}\frac1{t^3}\,.
$$
Hence
\be\lb{BoundI}
\int_{\bR^3}|II|^pM(v)dv=O\left(\frac{1}{t^3}\right)\,.
\ee

On the other hand
$$
\ba
\int_{\bR^3}|I|^pM(v)dv\le\int_{\bR^3}\frac{C\indc_{t>\tau_{x,v}}}{(t-\tau_{x,v}+1)^p}M(v)dv
\\
\le\int_{|v|<\eps}\frac{C}{((t-\frac2{|v|})_++1)^p}M(v)dv+\int_{|v|\ge\eps}\frac{C}{((t-\frac2{|v|})_++1)^p}M(v)dv
\\
\le C\int_{|v|<\eps}M(v)dv+\frac{C}{((t-\frac2\eps)_++1)^p}\int_{|v|\ge\eps}M(v)dv
\\
\le C'\eps^3+\frac{C'}{(\frac{t}2+1)^p}
\ea
$$
for some other constant $C'>0$, provided that $t\ge 4/\eps$. Choosing $\eps=4/t<1$, we arrive at
\be\lb{BoundII}
\int_{\bR^N}|I|^pM(v)dv\le\frac{4^3C'}{t^3}+\frac{C'}{t^p}\,.
\ee
Putting together (\ref{BoundI})-(\ref{BoundII}) entails the first estimate in the theorem. The second follows from the first and interpolation with
the maximum principle for the solution $g$ of (\ref{BVPSpher/M}).

Finally, the fact that
$$
\sqrt{2\pi}\mu_\infty=\tfrac1{|\Om|}\iint_{\Om\times\bR^3}f^{in}(|x|,|v|)dxdv
$$
follows from the conservation of mass for the solution of (\ref{BVPSpher}).
\end{proof}

\smallskip
Notice that the decay requirements on the renewal kernel in Feller's Theorem 4 \cite{Feller41} forbid extending our proof to space dimensions
less than $3$. In the one-dimensional case, Yu obtained in \cite{Yu09} an algebraic rate of convergence for the outgoing flux (specifically, of
order $t^{-1/10}$), by probabilistic arguments significantly more involved than the above proof.

\section{The monokinetic case: exponential decay}

The discussion in section \ref{S-NegResult} shows the existence of too many slow particles is the reason for the slow return to equilibrium in the 
case of a collisionless gas in a container with constant wall temperature. In the present section, we confirm this effect by proving that monokinetic
populations of particles in the same conditions converge to equilibrium exponentially fast.

Radiative transfer provides a natural example of a monokinetic transport process. Consider the following (academic) problem: at time $0$, a source
emits radiation in an empty container surrounded by infinitely opaque material. Photons propagate at the speed of light, and are instantaneously
thermalized on the wall. Denoting $I\equiv I(t,x,\om,\nu)$ the radiative intensity at time $t$, position $x$, in the direction $\om$ and with frequency
$\nu$, one arrives at the following initial-boundary value problem:
\be\lb{BVPRad}
\left\{
\begin{array}{ll}
\tfrac1c\d_t I+\om\cdot\grad_xI=0\,,&\quad x\in\Om\,,\,\,|\om|=1\,,\,\,\nu>0\,,
\\	\\
I(t,x,\om,\nu)=B_\nu(\th_I(t,x))\,,&\quad x\in\d\Om\,,\,\,\om\cdot n_x<0\,,\,\,\nu>0\,,
\\	\\
I\rstr_{t=0}=I^{in}\,.&
\end{array}
\right.
\ee
Here $c>0$ is the speed of light in vacuum, $\Om$ is a bounded open set in $\bR^3$ with $C^1$ boundary, $n_x$ is the outward unit normal at
$x\in\d\Om$. The notation $B_\nu(\th)$ designates Planck's function for the radiative intensity emitted by a black body at temperature $\th$, i.e.
$$
B_\nu(\th)=\frac{2h\nu^3}{c^2}\frac{1}{e^{h\nu/k\th}-1}\,,
$$
where $h$ and $k$ are respectively the Planck and the Boltzmann constants. The temperature $\th_{I}$ is defined by
$$
\ba
\iint_{\bS^2\times\bR_+}&I(t,x,\om,\nu)(\om\cdot n_x)_+d\om d\nu
\\
&=\iint_{\bS^2\times\bR_+}B_\nu(\th_I(t,x))(\om\cdot n_x)_+d\om d\nu=:\si\th_I(t,x)^4
\ea
$$
where $\si$ is the Stefan-Boltzmann constant 
$$
\si:=\frac{2\pi k^4}{c^2h^3}\int_0^\infty\frac{x^3dx}{e^x-1}=\frac{2\pi^5k^4}{15c^2h^3}\,.
$$
By definition of $\th_I$, the net flux of energy at each point of the wall is 
$$
\iint_{\bS^2\times\bR_+}I(t,x,\om,\nu)\om\cdot n_xd\om d\nu=0\,,
$$
so that the total energy in the container is conserved: for each $t\ge 0$, one has
$$
\iiint_{\Om\times\bS^2\times\bR_+}I(t,x,\om,\nu)dxd\om d\nu=\iiint_{\Om\times\bS^2\times\bR_+}I^{in}(x,\om,\nu)dxd\om d\nu\,.
$$
Therefore, one expects that, as $t\to+\infty$, the radiation inside the container thermalizes so that
$$
I(t,x,\om,\nu)\to B_\nu(\th_\infty)\quad\hbox{Êas }t\to+\infty\,,
$$
with
$$
4\si|\Om|\th_\infty^4:=4\pi|\Om|\int_0^\infty B_\nu(\th_\infty)d\nu=\iiint_{\Om\times\bS^2\times\bR_+}I^{in}(x,\om,\nu)dxd\om d\nu\,.
$$
In this section, we seek the rate of convergence to equilibrium for the problem above.

\begin{Thm}\lb{T-ExpDecayRad}
Assume that $I^{in}$ is smooth and satisfies $\Supp(I^{in})\subset\Om\times\bR_+$ together with the bound $0\le I^{in}(x,\nu)\le B_\nu(\Theta)$
with $\Theta>0$, for all $x\in\Om$ and $\nu\ge 0$. Then, the solution $I$ of the initial boundary value problem (\ref{BVPRad}) satisfies
$$
I(t,x,\om,\nu)-B_\nu(\th_\infty)=O(e^{-\a't})
$$
as $t\to+\infty$, uniformly on $\Om\times\bS^1\times\bR_+$, for each $\a'<\a$, where $\th_\infty\ge 0$ is given by
$$
\th_\infty^4:=\frac{\pi}{\si|\Om|}\iint_{\Om\times\bS^2}I^{in}(|x|^2,\nu)dxd\nu
$$
while $\a$ is
$$
\a=\inf\left\{-\Re(z)\,|\,z\in\bC^*\,\hbox{ such that }\,\tfrac12\int_0^2 e^{-sz}sds=1\right\}\,.
$$
\end{Thm}

\begin{proof}
The problem is simplified by considering the radiative intensity integrated over frequencies, the grey intensity $u$ defined as
$$
u(t,x,\om):=\int_0^\infty I(t,x,\om,\nu)d\nu\,.
$$
The formula defining $\th_I$ reads
$$
\int_{\bS^2}u(t,x,\om)(\om\cdot n_x)_+d\om=\si\th_I(t,x)^4
$$
while
$$
\int_0^\infty B_\nu(\th_I)d\nu=\frac{\si}{\pi}\th_I^4\,.
$$
Therefore, choosing for simplicity units of time and space so that $c=1$, the grey intensity $u$ satisfies
\be\lb{BVPGreyRad}
\left\{
\begin{array}{ll}
\d_t u+\om\cdot\grad_xu=0\,,&\quad x\in\Om\,,\,\,|\om|=1\,,
\\	\\
u(t,x,\om)=\tfrac1\pi\displaystyle\int_{\bS^2}u(t,x,\om')(\om'\cdot n_x)_+d\om'&\quad x\in\d\Om\,,\,\,\om\cdot n_x<0\,,
\\	\\
u\rstr_{t=0}=u^{in}:=\displaystyle\int_0^\infty I^{in}d\nu\,.&
\end{array}
\right.
\ee

Henceforth, we assume the same geometric setting as in section \ref{S-AlgDecay}: $\Om=B(0,1)\subset\bR^3$, and $I^{in}$ is radial in $x$ and
isotropic, meaning that $I^{in}\equiv I^{in}(|x|^2,\nu)$. In that case, reasoning as in Lemma \ref{L-BoundarySym} shows that 
$$
f(t):=\tfrac1\pi\displaystyle\int_{\bS^2}u(t,x,\om')(\om'\cdot n_x)_+d\om'
$$ 
is independent of the position $x\in\d\Om$. 

Proceeding as in section \ref{S-AlgDecay}, we derive a renewal integral equation for $f$. Indeed, integrating the transport equation in
(\ref{BVPGreyRad}) by the method of characteristics, we see that
\be\lb{CharGreyRad}
u(t,x,\om)=\indc_{t\le\tau_{x,\om}}u^{in}(|x-t\om|^2)+\indc_{t>\tau_{x,\om}}f(t-\tau_{x,\om})
\ee
whenever $x\in\d\Om$ with $\om\cdot n_x>0$, so that
$$
\ba
f(t)&=\tfrac1\pi\int_{\bS^2}u(t,x,\om)(\om\cdot n_x)_+d\om
\\
&=\tfrac1\pi\int_{\bS^2}\indc_{t\le\tau_{x,\om}}u^{in}(|x-t\om|^2)(\om\cdot n_x)_+d\om
\\
&+
\tfrac1\pi\int_{\bS^2}\indc_{t>\tau_{x,\om}}f(t-\tau_{x,\om})(\om\cdot n_x)_+d\om
\\
&=S(t)+\int_0^tK(s)f(t-s)ds\,.
\ea
$$
Here
$$
\ba
\int_0^tK(s)f(t-s)ds
=
\tfrac1\pi\int_{\bS^2}\indc_{t>\tau_{x,\om}}f(t-\tau_{x,\om})(\om\cdot n_x)_+d\om
\\
=
\tfrac1\pi\int_0^{\pi/2}\indc_{t>2\cos\th}f(t-2\cos\th)\cos\th\cdot2\pi\sin\th d\th
\\
=
\int_0^2\indc_{t>s}f(t-s)\tfrac12sds
\\
=\int_0^tK(s)f(t-s)ds
\ea
$$
so that
$$
K(s)=\tfrac12 s\indc_{s<2}\,,
$$
while
$$
\ba
S(t):=\tfrac1\pi\int_{\bS^2}\indc_{t\le\tau_{x,\om}}u^{in}(|x-t\om|^2)(\om\cdot n_x)_+d\om
\\
=\tfrac1\pi\int_{\bS^2}\indc_{t\le\tau_{x,\om}}u^{in}(1+t^2-2tx\cdot\om)(\om\cdot n_x)_+d\om
\\
=\tfrac1\pi\int_0^{\pi/2}\indc_{t\le 2\cos\th}u^{in}(1+t^2-2t\cos\th)\cos\th\cdot 2\pi\sin\th d\th
\\
=\tfrac12\int_0^2\indc_{t\le s}u^{in}(1+t^2-ts)sds
\\
=\int_0^{+\infty}u^{in}(1-t\tau)K(t+\tau)d\tau\,.
\ea
$$

Since $u^{in}\in L^\infty(\Om)$, and $K$ is bounded with support in $[0,2]$, so is $S$ and one has
$$
\int_0^\infty t^mK(t)dt<+\infty\quad\hbox{ and }S(t)=O(t^{-m}) \hbox{ as }t\to+\infty \hbox{ for each }m\ge 1\,.
$$
Applying Theorem 4 of \cite{Feller41} already implies that 
$$
f(t)\to f_\infty:=\frac{\displaystyle\int_0^\infty S(t)dt}{\displaystyle\int_0^\infty tK(t)dt}\quad\hbox{ as }t\to+\infty\,,
$$
and that $f(t)-f_\infty$ is rapidly decaying as $t\to+\infty$:
$$
|f(t)-f_\infty|=O(t^{-m})\quad\hbox{ as }t\to+\infty\hbox{ for each }m\ge 1\,.
$$

We improve this estimate and obtain exponential convergence, as follows. Given $\phi\in L^2(\bR_+)$, we denote its Laplace transform by
$$
\tilde\phi(z):=\int_0^\infty e^{-zt}\phi(t)dt\,.
$$
We recall that the Laplace transform $\phi\mapsto\tilde\phi$ is an isomorphism from $L^2(\bR_+)$ to the Hardy space $\cH^2(\Pi_0)$, where
$\Pi_a:=\{z\in\bC\,|\,\Re(z)>a\}$. We recall that an element of $\cH^2(\Pi_a)$ is a holomorphic function $\psi\equiv\psi(z)$ on the half-plane 
$\Pi_a$ that satisfies
$$
\sup_{x>a}\int_\bR|\psi(x+iy)|^2dy<+\infty\,.
$$

Applying the Laplace transform to both sides of the renewal equation
\be\lb{RenewRad}
f(t)=S(t)+\int_0^tK(s)f(t-s)ds\,,\qquad t\ge 0
\ee
satisfied by $f$, we see that
\be\lb{RenewLap}
\tilde f(z)(1-\tilde K(z))=\tilde S(z)\,,\quad\Re(z)>0\,.
\ee

Since $K\in L^\infty(\bR)$ with support in $[0,2]$, the Laplace transform of $K$ is entire holomorphic; it is given by
$$
\tilde K(z)=\frac{1-e^{-2z}(1+2z)}{2z^2}\,,\quad z\not=0\,,\qquad \tilde K(0)=1\,.
$$
Since $\tilde K'(0)=-\tfrac43$, the only zero of $1-\tilde K$ near $z=0$ is $0$. Besides
$$
|\tilde K(z)|=\int_0^\infty e^{-t\Re(z)}K(t)dt<1\qquad\hbox{ whenever }\Re(z)>0\,,
$$
since $K$ is a probability distribution on $\bR_+$. Furthermore
$$
\Re(\tilde K(iy))=\int_0^\infty\cos(ty)K(t)dt<\int_0^\infty K(t)dt=1
$$
so that $z=0$ is the only zero of $1-K$ on the imaginary axis. Finally, for each $a\in\bR$, one has
\be\lb{Estim|LapK|}
|\tilde K(z)|\le\frac1{2|z|^2}(1+(1+2|z|)e^{-2\Re(z)})\,.
\ee
In particular $|\tilde K(z)|\to 0$ as $\Im(z)\to+\infty$ uniformly in $\Re(z)\ge a$. Therefore, there exists $\a>0$ such that 
$$
1-\tilde K(z)=0\hbox{ and }\Re(z)>-\a\quad\Leftrightarrow\quad z=0\,.
$$
Hence $z\mapsto\frac{z}{1-\tilde K(z)}$ is holomorphic on the half-plane $\Pi_{-\a}$, and one recasts (\ref{RenewLap}) as
$$
z\tilde f(z)=\frac{z\tilde S(z)}{1-\tilde K(z)}\,,\quad\Re(z)>-\a\,.
$$
Indeed, since $S\in L^\infty(\bR)$ with support in $[0,2]$, its Laplace transform is an entire holomorphic function. 

Moreover, assuming that $u^{in}\in C^1_c((0,1))$, we deduce from the formula
$$
S(t)=\int_t^\infty u^{in}(1+t^2-ts)K(s)ds
$$
that $S\in C^1(\bR_+)$, with
$$
S'(t)=\int_t^\infty (u^{in})'(1+t^2-ts)(2t-s)K(s)ds
$$
so that, for all $t\ge 0$, one has
$$
|S'(t)|\le 2t\|(u^{in})'\|_{L^\infty}\int_t^\infty K(s)ds\le 4\|(u^{in})'\|_{L^\infty}\,.
$$
Since $S(0)=u^{in}(1)=0$, one has $z\tilde S(z)=\widetilde{S'}(z)$; and since $S'$ is continuous with compact support on $\bR$, its Laplace
transform $\widetilde{S'}\in\cH^2(\Pi_a)$ for each $a\in\bR$. 

In view of (\ref{Estim|LapK|}), for each $\a'<\a$, the function $z\mapsto(1-\tilde K(z))^{-1}$ is bounded on $\Pi_{-\a'}\setminus B(0,\a'/2)$, so that 
the function 
$$
z\mapsto\frac{z\tilde S(z)}{1-\tilde K(z)}\hbox{ belongs to }\cH^2(\Pi_{-\a'})
$$
for each $\a'<\a$. 

Since $\widetilde{e^{at}\phi}=\tilde\phi(z-a)$, this implies the existence of a function $g$ such that $e^{\a't}g\in L^2(\bR_+)$ for each $\a'<\a$ 
and
$$
\tilde g(z)=\frac{z\tilde S(z)}{1-\tilde K(z)}\,,\quad\hbox{ whenever }\Re(z)>-\a'\,.
$$
Consider the function $G$ defined by
$$
G(t)=\int_0^tg(s)ds\,,\quad t\ge 0\,.
$$
Then $G$ is continuous on $\bR$ with support in $\bR_+$, satisfies 
$$
0\le|G(t)|\le\frac{\|e^{\a't}g\|_{L^2}}{\sqrt{2\a'}}\,,\quad t\ge 0\,,
$$
so that $\tilde G$ belongs to $\cH^2(\Pi_\eta)$ for each $\eta>0$. 

Now, for each $z\in\Pi_\eta$
$$
\ba
z\tilde G(z)&=\int_0^\infty ze^{-zt}\int_0^tg(s)dsdt=\int_0^\infty g(s)\int_s^\infty ze^{-zt}dtds
\\
&=\int_0^\infty g(s)e^{-zs}ds=\tilde g(z)=\frac{z\tilde S(z)}{1-\tilde K(z)}=z\tilde f(z)\,.
\ea
$$
In other words, $G$ and $f$ are bounded continuous functions on $\bR$ supported in $\bR_+$, whose Laplace transforms coincide on $\Pi_0$.
Therefore 
$$
f(t)=G(t)=\int_0^tg(s)ds\,,\qquad\hbox{ for each }t\ge 0\,,
$$
and since $e^{\a't}g\in L^2(\bR_+)$ for each $\a'<\a$, one conclude that
$$
\ba
|f(t)-f_\infty|&=\int_t^{+\infty}|g(s)|ds\le\|e^{\a' t}g\|_{L^2}\left(\int_t^{+\infty}e^{-\a' s}ds\right)^{1/2}
\\
&=\frac{\|e^{\a' t}g\|_{L^2}}{\sqrt{2\a'}}e^{-\a' t}\le C_{\a'}e^{-\a' t}
\ea
$$
for each $t>0$.

Going back to the original problem (\ref{BVPRad}) or (\ref{BVPGreyRad}), we deduce from the method of charateristics (\ref{CharGreyRad}) that
$$
u(t,x,\om)=f(t)\,,\quad\hbox{ whenever }t>\hbox{diam}(\Om)\,,
$$
or equivalently, in the frequency dependent case,
$$
I(t,x,\om,\nu)=B_\nu\left((\tfrac{\pi}{\si}f(t))^{1/4}\right)\,,\quad\hbox{ whenever }t>\frac{\hbox{diam}(\Om)}{c}\,,
$$
since $\tau_{x,\om}\le\hbox{diam}(\Om)$. 
\end{proof}

\smallskip
Before concluding, it is worth discussing the role of the monokinetic assumption in the exponential decay estimate obtained here.

Intrinsically, the integral kernel $K(s)ds$ in the proof of Theorem \ref{T-ExpDecayRad} is the image under $\om\mapsto\tau_{x,\om}$ of the 
probability measure proportional to $(\om\cdot n_x)_+d\om$ on $\bS^2$. Since $0\le\tau_{x,\om}\le 2=\hbox{diam}(\Om)$, this kernel $K$ is 
supported in $[0,2]$. Therefore its Laplace transform is holomorphic entire, and since solving the renewal integral equation is equivalent to 
the possibility of inverting $1-\tilde K$, the only constraint is the potential presence of zeros of $1-\tilde K$; besides, since $K$ has integral
$1$, the half-plane $\Re(z)\ge 0$ is zero-free, except for $z=0$, and zeros of $1-\tilde K$ cannot accumulate on the imaginary axis. Hence
the solution of the renewal equation is obtained by inverting the Laplace transform, deforming the contour of integration to be $\Re(z)=-\a'$ 
for some $\a'>0$.

In the case of the collisionless gas in Theorem \ref{T-AlgDecayBoltz}, the integral kernel $K(s)ds$ is the image under $v\mapsto\tau_{x,v}$ 
of the probability measure proportional to $(v\cdot n_x)_+M(v)dv$ on $\bR^3$. Because of the low-speed gas molecules, its behavior near
$s=0$ is the same as that of the image measure of, say, $(v_1)_+dv$ on the unit ball $B(0,1)$ under $v\mapsto\tau_{x,v}$ --- the Maxwellian
weight $M(v)$ obviously plays no role near $v=0$. Thus, the decay of $K$ for $s\to+\infty$ only depends on the Jacobian weight $|v|^{N-1}$
that appears when integrating radial functions with the Lebesgue measure in $\bR^N$. This results in only algebraic decay for $K$ in the
limit $s\to+\infty$, so that the Laplace transform of $K$ does not have a holomorphic extension in any half-plane of the form $\Re(z)>-\a$ with
$\a>0$. Therefore, the same deformation of contour as in the proof of Theorem \ref{T-ExpDecayRad} is impossible in that case. In order to
be able to push the integration contour in the Laplace inversion formula to the left of the imaginary axis, $K$ should decay exponentially fast
as $s\to+\infty$. This would be the case if the Jacobian weight $|v|^{N-1}$ appearing in the integration of radial functions with the Lebesgue 
measure in space dimension $N$ could be replaced with $e^{-a/|v|}$ for some $a>0$. Thus, the difference between the speed of approach
to equilibrium observed in a collisionless gas with diffuse reflection of gas molecules on vessel walls kept at constant temperature, and the
spectral gap property for the linear Boltzmann equation on the other hand, can be measured in terms of the vanishing order of $|v|^{N-1}$ 
and $e^{-a/|v|}$ as $|v|\to 0$.

\section{Conclusion}

We have discussed in this paper the speed of approach to equilibrium for a collisionless gas enclosed in a vessel whose wall is kept at a uniform 
and constant temperature, assuming diffuse reflection on the vessel wall. We propose lower bounds for this convergence rate, depending on the 
possibility of initial distribution functions that are concentrated at low particle speeds. Not only do these lower bounds exclude exponential 
decay to equilibrium, they also constrain the power in any potential algebraic decay rate. On the other hand, assuming spherical symmetry of 
the initial velocity distribution function and of the vessel, we establish an algebraic decay rate by reducing the problem to a renewal equation
for the flux of outgoing gas particles. All these results are consistent with the asymptotic behavior obtained in \cite{TsujiAokiFG2010} with
numerical simulations. To clarify the role of low-speed particles, we have studied the case of radiative transfer in the vacuum, in a container
with infinitely opaque boundary, so that the total incoming radiation flux on the boundary is reemitted as a Planck distribution. We establish
exponential convergence of the radiative intensity to a Planck distribution in the long time limit, assuming the same spherical symmetries as
in the case of a collisionless gas.

There remain several open problems in this direction, such as a) establishing decay rates without any symmetry assumptions, and b) obtaining
the optimal decay rates, even in the spherically symmetric case. Both problems might involve mathematical techniques quite different from the
ones used here. We hope to return to these questions in subsequent publications.


\bigskip
\noindent
\textbf{Acknowledgements.}  We express our gratitude to Claude Bardos and Laurent Desvillettes for valuable discussions about the problem
considered here, and to the Isaac Newton Institute for Mathematical Sciences for its support and hospitality during the preparation of 
this paper. This work was partially supported by grant-in-aid for scientific research No. 20360046 from JSPS.


\end{document}